\documentclass[11pt]{article}
\usepackage[margin=1in]{geometry}

\usepackage{multicol}

\usepackage{graphicx} % for pdf, bitmapped graphics files
\usepackage{amsmath} % assumes amsmath package installed
\usepackage{amssymb}  % assumes amsmath package installed
\usepackage{amsthm}

\usepackage{scalefnt}
\usepackage{hyperref} % Amirali: I commented this package out because after 22 equations, it wasn't
\usepackage{lscape}  %AAA: adding this for a large matrix in the appendix.
\usepackage{color}
\usepackage{subfigure}

\usepackage{comment}

\makeatletter
\renewcommand*\env@matrix[1][c]{\hskip -\arraycolsep
  \let\@ifnextchar\new@ifnextchar
  \array{*\c@MaxMatrixCols #1}}
\makeatother

\theoremstyle{plain}  % Bold name, italics font
\newtheorem{theorem}{Theorem}[section]
\newtheorem{lemma}[theorem]{Lemma}
\newtheorem{proposition}[theorem]{Proposition}

\theoremstyle{definition}

\theoremstyle{remark} % italics name, roman font
\newtheorem{example}{Example}[section]
\newtheorem{remark}{Remark}[section]

\definecolor{Pink}{RGB}{225,0,127}
\title{\LARGE \bf Sum of squares certificates for stability of\\ planar, homogeneous, and switched systems}

 \author{Amir Ali Ahmadi and Pablo A. Parrilo \thanks{Amir Ali Ahmadi is with the Department of Operations Research and Financial Engineering at Princeton University. Pablo A. Parrilo is with the Department of Electrical Engineering and Computer Science at MIT.\newline Email: \texttt{a\_a\_a@princeton.edu}, \texttt{parrilo@mit.edu}.}
}

\begin{document}
\date{}
\maketitle
%

% The paper headers
%\markboth{Journal of \LaTeX\ Class Files,~Vol.~11, No.~4, December~2012}%
%{Shell \MakeLowercase{\textit{et al.}}: Bare Demo of IEEEtran.cls for Journals}
% The only time the second header will appear is for the odd numbered pages
% after the title page when using the twoside option.
% 
% *** Note that you probably will NOT want to include the author's ***
% *** name in the headers of peer review papers.                   ***
% You can use \ifCLASSOPTIONpeerreview for conditional compilation here if
% you desire.

% If you want to put a publisher's ID mark on the page you can do it like
% this:
%\IEEEpubid{0000--0000/00\$00.00~\copyright~2012 IEEE}
% Remember, if you use this you must call \IEEEpubidadjcol in the second
% column for its text to clear the IEEEpubid mark.

% use for special paper notices
%\IEEEspecialpapernotice{(Invited Paper)}

% make the title area
\maketitle

% As a general rule, do not put math, special symbols or citations
% in the abstract or keywords.
\begin{abstract}
%%In view of the pervasiveness of sum of squares (sos) techniques in
%Lyapunov analysis of dynamical systems, the converse question of
%whether sos Lyapunov functions exist whenever polynomial Lyapunov
%functions exist deserves more attention. In this paper, we present a few results in this direction for the study of global asymptotic stability of polynomial vector fields and switched linear systems. We first show
%via an explicit counterexample that if the degree of the
%polynomial Lyapunov function is fixed, then sos programming can
%fail to find a valid Lyapunov function even though one exists. On
%the other hand, if the degree is allowed to increase,
%
%We prove
%We show that existence of a polynomial Lyapunov function for a homogeneous or a planar polynomial vector field implies existence of a polynomial Lyapunov
%function that is sos and that the negative of its derivative is
%also sos. This result is extended to develop a converse sos
%Lyapunov theorem for stability of switched linear systems. For this class of systems, we further show that if the derivative inequality of the Lyapunov function has an sos certificate, then the Lyapunov function is automatically a sum of squares. Finally, we demonstrate some merits of replacing the sos requirement on a polynomial Lyapunov function by an sos requirement on its top homogeneous component.
We show that existence of a global polynomial Lyapunov function for a homogeneous polynomial vector field or a planar polynomial vector field (under a mild condition) implies existence of a polynomial Lyapunov
function that is a sum of squares (sos) and that the negative of its derivative is
also a sum of squares. This result is extended to show that such sos-based certificates of stability are guaranteed to exist for all stable switched linear systems. For this class of systems, we further show that if the derivative inequality of the Lyapunov function has an sos certificate, then the Lyapunov function itself is automatically a sum of squares. These converse results establish cases where semidefinite programming is guaranteed to succeed in finding proofs of Lyapunov inequalities. Finally, we demonstrate some merits of replacing the sos requirement on a polynomial Lyapunov function with an sos requirement on its top homogeneous component. In particular, we show that this is a weaker algebraic requirement in addition to being cheaper to impose computationally.

\end{abstract}

% Note that keywords are not normally used for peerreview papers.
%\begin{IEEEkeywords}
%IEEEtran, journal, \LaTeX, paper, template.
%\end{IEEEkeywords}

% For peer review papers, you can put extra information on the cover
% page as needed:
% \ifCLASSOPTIONpeerreview
% \begin{center} \bfseries EDICS Category: 3-BBND \end{center}
% \fi
%
% For peerreview papers, this IEEEtran command inserts a page break and
% creates the second title. It will be ignored for other modes.

%\IEEEpeerreviewmaketitle

\section{Introduction}
Consider a continuous time dynamical system
\begin{equation}\label{eq:CT.dynamics}
\dot{x}=f(x),
\end{equation}
where $f:\mathbb{R}^n\rightarrow\mathbb{R}^n$ is a polynomial and
has an equilibrium at the origin, i.e., $f(0)=0$. When a
polynomial function $V(x):\mathbb{R}^n\rightarrow\mathbb{R}$ is
used as a candidate Lyapunov function for stability analysis of
system (\ref{eq:CT.dynamics}), conditions of Lyapunov's theorem
reduce to a set of polynomial inequalities. For instance, if
establishing global asymptotic stability of the origin is desired (see, e.g.,~\cite[Chap. 4]{Khalil:3rd.Ed} for a formal definition),
one would require a radially unbounded polynomial Lyapunov
function candidate to satisfy:
\begin{eqnarray}
V(x)&>&0\quad \forall x\neq0 \label{eq:V.positive} \\
\dot{V}(x)=\langle\nabla V(x),f(x)\rangle&<&0\quad \forall x\neq0.
\label{eq:Vdot.negative}
\end{eqnarray}
Here, $\dot{V}$ denotes the time derivative of $V$ along the
trajectories of (\ref{eq:CT.dynamics}), $\nabla V$ is the
gradient vector of $V$, and $\langle .,. \rangle$ is the standard
inner product in $\mathbb{R}^n$. In some other variants of the
analysis problem, e.g. if LaSalle's invariance principle is to be
used, or if the goal is to prove boundedness of trajectories of
(\ref{eq:CT.dynamics}), then the inequality in
(\ref{eq:Vdot.negative}) is replaced with
\begin{equation}\label{eq:Vdot.nonpositive}
\dot{V}(x)\leq 0 \quad \forall x.
\end{equation}
In either case, the problem arising from this analysis approach is
that even though polynomials of a given degree are finitely
parameterized, the computational problem of searching for a
polynomial $V$ satisfying inequalities of the type
(\ref{eq:V.positive}),(\ref{eq:Vdot.negative}),(\ref{eq:Vdot.nonpositive})
is intractable. In fact, even deciding if a given polynomial $V$
of degree four or larger satisfies (\ref{eq:V.positive}) is
NP-hard~\cite{nonnegativity_NP_hard},~\cite[Prop. 1]{AAA_ACC12_Cubic_Difficult}.

An approach pioneered in~\cite{PhD:Parrilo} and quite popular by
now is to replace the positivity or nonnegativity conditions by
the requirement of existence of a sum of squares (sos)
decomposition (see Section~\ref{sec:prelim} for a definition):
\begin{eqnarray}
V& \mbox{sos}\label{eq:V.SOS} \\
-\dot{V}=-\langle\nabla V,f\rangle& \mbox{sos}.
\label{eq:-Vdot.SOS}
\end{eqnarray}

Clearly, if a polynomial is a sum of squares of other polynomials, then it must be nonnegative. Moreover, it is well known that an sos decomposition constraint on a polynomial can be cast as a semidefinite programming (SDP) problem~\cite{sdprelax}, which can be solved efficiently.
% on existence of an sos decomposition can be cast as a semidefinite program (SDP)~\cite{sdprelax}, which can be solved efficiently.
%, e.g. by using interior point
%algorithms~\cite{VaB:96}.
\footnote{If the SDP resulting from (\ref{eq:V.SOS}) and (\ref{eq:-Vdot.SOS}) is strictly feasible, then any interior solution automatically satisfies (\ref{eq:V.positive})-(\ref{eq:Vdot.negative}); see, e.g.,~\cite[p. 41]{AAA_MS_Thesis}.} 

%We call a polynomial Lyapunov function that in addition to (\ref{eq:V.positive})-(\ref{eq:Vdot.negative}) satisfies both sos conditions in (\ref{eq:V.SOS}) and (\ref{eq:-Vdot.SOS}) a
%\emph{sum of squares Lyapunov function}.\footnote{If the SDP resulting from (\ref{eq:V.SOS}) and (\ref{eq:-Vdot.SOS}) is strictly feasible, then any interior solution automatically satisfies (\ref{eq:V.positive})-(\ref{eq:Vdot.negative}); see, e.g.,~\cite[p. 41]{AAA_MS_Thesis}.} We emphasize that this is
%the sensible definition of a sum of squares Lyapunov function and
%not what the name may suggest, which is a Lyapunov function that
%is a sum of squares. Indeed, the underlying semidefinite program
%will find a Lyapunov function $V$ only if $V$ satisfies
%\emph{both} conditions (\ref{eq:V.SOS}) and (\ref{eq:-Vdot.SOS}).

%if a Lyapunov function $V$ satisfies (\ref{eq:V.SOS}) but not the
%derivative condition (\ref{eq:-Vdot.SOS}), then the semidefinite
%program will \emph{not} find this Lyapunov function.

Over the last decade, Lyapunov analysis with sum of squares techniques has become a relatively well-established approach for a variety of problems in controls. Examples include stability analysis of switched and hybrid systems, design of nonlinear controllers, and formal verification of safety-critical systems, just to name a few; % and a
%multitude of sos techniques have risen to tackle a broad range
%of problems in systems and control.
see, e.g.,~\cite{PositivePolyInControlBook},
\cite{AutControlSpecial_PositivePolys}, \cite{ControlAppsSOS}, \cite{PraP03}, \cite{PapP02}, and references therein.
%
%see, e.g.,~\cite{PositivePolyInControlBook},
%\cite{AutControlSpecial_PositivePolys}, \cite{ControlAppsSOS}, \cite{Peet_sample_data}, \cite{PraP03}, \cite{PapP02}, \cite{Pablo_Rantzer_synthesis},
%\cite{Chest.et.al.sos.robust.stability}, \cite{Erin_Pablo_Contraction}, \cite{Tedrake_LQR_Trees}, and references therein.
%
%
%\cite{PositivePolyInControlBook},
%\cite{AutControlSpecial_PositivePolys}, \cite{Chesi_book},
%\cite{ControlAppsSOS}, \cite{PraP03}, \cite{PapP02},
%\cite{Pablo_Rantzer_synthesis},
%\cite{Chest.et.al.sos.robust.stability},
%\cite{AAA_PP_CDC08_non_monotonic}, \cite{Erin_Pablo_Contraction},
%\cite{Tedrake_LQR_Trees}
%Despite the wealth of research in this area, the converse question
%of whether the existence of a polynomial Lyapunov function implies
%the existence of a Lyapunov function satisfying the sos conditions
%in (\ref{eq:V.SOS}), (\ref{eq:-Vdot.SOS}) has remained unresolved.
Despite the wealth of research in this area, the literature by and large focuses on proposing the sum of squares constraints as a sufficient condition for the underlying Lyapunov inequalities, without studying their necessity. For example, even for the basic notion of global asymptotic stability (GAS), the following question is to the best of our knowledge open:

%following ``converse'' question has not been fully understood, even for the basic notion of global asymptotic stability: 

\noindent {\bf Problem 1:} Does existence of a polynomial Lyapunov function satisfying (\ref{eq:V.positive})-(\ref{eq:Vdot.negative}) imply existence of a polynomial Lyapunov function (of possibly higher degree) satisfying (\ref{eq:V.SOS})-(\ref{eq:-Vdot.SOS})?

%In this paper, 

We consider this question for polynomial vector fields, as well as switched systems, and provide a positive answer in some special cases. (We also study other related questions.) Before we outline our contributions, some remarks about the statement of Problem 1 are in order. 
% study this problem and some related questions for polynomial vector fields and switched linear systems. 

%First, we point out that between the two sos requirements in (\ref{eq:V.SOS})-(\ref{eq:-Vdot.SOS}), the condition in (\ref{eq:-Vdot.SOS}) is what is often more challenging to ensure as the sos requirement in (\ref{eq:V.SOS}) can always be met simply by squaring a polynomial Lyapunov function (see Lemma~\ref{lem:W=V^2}). (In fact, we show that for some systems, the sos requirement in (\ref{eq:V.SOS}) comes for free if the condition in (\ref{eq:-Vdot.SOS}) is met; see Subsection~\ref{subsec:switched.V.always.sos}.)

First, we point out that between the two sos requirements in (\ref{eq:V.SOS})-(\ref{eq:-Vdot.SOS}), the condition in (\ref{eq:V.SOS}) can simply be met by squaring a polynomial Lyapunov function (see Lemma~\ref{lem:W=V^2}). By contrast, the condition in (\ref{eq:-Vdot.SOS}) is typically more challenging to ensure. Second, we remark that imposing the sos conditions in (\ref{eq:V.SOS})-(\ref{eq:-Vdot.SOS}) is not the only way to use the sos relaxation for proving GAS of (\ref{eq:CT.dynamics}). For example, even when these polynomials are not sos, one can multiply them by a positive polynomial (e.g., a power of $\sum x_i^2$) and the result may become sos and then certify the desired inequalities. In this paper, however, we are interested in knowing whether the sos conditions on a Lyapunov function and its derivative can be met just by increasing the degree of the Lyapunov function. This is a very basic question in our opinion and it is in fact how the sos relaxation is most commonly used in practice. Finally, since our interest is mainly in establishing GAS of (\ref{eq:CT.dynamics}), a natural question that comes before Problem 1 is the following: 
%``Does global asymptotic stability of a polynomial vector field imply existence of a polynomial Lyapunov function  (i.e., a polynomial satisfying conditions (\ref{eq:V.positive})-(\ref{eq:Vdot.negative}))?''
``Does global asymptotic stability of a polynomial vector field imply existence of a polynomial Lyapunov function satisfying (\ref{eq:V.positive})-(\ref{eq:Vdot.negative})?''
%
%Our focus in this paper is not on this question, though we should mentio
The answer to this question is in general negative. The interested reader can find explicit counterexamples in~\cite{Bacciotti.Rosier.Liapunov.Book},~\cite{AAA_MK_PP_CDC11_no_Poly_Lyap}.

\subsection{Contributions}
%We start in Section~\ref{sec:prelim} by recalling a few basic definitions.
% and
%reviewing some known results about the relationship between nonnegative and sum
%of squares polynomials.
%In Section~\ref{sec:the.counterexample}, we answer {\bf Problem 1} in the negative by constructing a quintic GAS vector field that admits a quadratic Lyapunov function, but no quadratic function satisfying the sos constraints in (\ref{eq:V.SOS})-(\ref{eq:-Vdot.SOS}).
%On the other hand,

Many of the results in this paper can be seen as establishing cases where semidefinite programming is guaranteed to succeed in finding proofs of Lyapunov inequalities. More precisely, in Section~\ref{sec:converse.sos.results}, we
give a positive answer to Problem~1 in the case where the
vector field is homogeneous
%\footnote{See Section~\ref{sec:converse.sos.results} for the definition of a homogeneous vector field and related references in controls.} 
(Theorem~\ref{thm:poly.lyap.then.sos.lyap}) or when it is planar and an additional mild assumption is met (Theorem~\ref{thm:poly.lyap.then.sos.lyap.PLANAR}). The general case remains open. The proofs of
these two theorems are quite simple and rely on powerful and relatively recent  Positivstellens\"atze due to Scheiderer
(Theorems~\ref{thm:claus} and~\ref{thm:claus.3vars}).

In Section~\ref{subsec:extension.to.switched.sys}, we extend these
results for % derive a converse sos Lyapunov theorem for asymptotic stability of 
stability analysis of switched linear systems. These are a widely-studied subclass of hybrid systems. Our result combined with a result of Mason et al.~\cite{switch.common.poly.Lyap} shows that if such a system is asymptotically stable under arbitrary switching, then it admits a common polynomial Lyapunov function that is sos and that the
negative of its derivative is also sos (Theorem~\ref{thm:converse.sos.switched.sys}). We also show that
for switched linear systems (both in discrete and continuous
time), if the inequality on the decrease condition of a Lyapunov
function is satisfied with a sum of squares certificate, then the Lyapunov
function itself is automatically a sum of squares
(Propositions~\ref{prop:switch.DT.V.automa.sos}
and~\ref{prop:switch.CT.V.automa.sos}). This statement, however, is shown to be false for nonlinear systems (Lemma~\ref{lem:V.autom.sos.not.true.NL}).

Finally, in Section~\ref{sec.h.o.t}, we demonstrate a related curious fact, that instead of requiring a candidate polynomial Lyapunov function to be sos, it is better to ask its top homogeneous component to be sos. We show that this still implies GAS (Proposition~\ref{prop:hot.implies.gas} and Lemma~\ref{lem:hot.is.no.more.conservative}), is a weaker algebraic requirement (Lemma~\ref{lem:hot.is.no.more.conservative}), introduces no conservative in dimension two, and is cheaper to impose. % computationally. 

\subsection{Related literature}

In related work~\cite{Peet.exp.stability},~\cite{Peet.Antonis.converse.sos.journal}, Peet and Papachristodoulou study similar questions for the notion of exponential stability on compact sets. In~\cite{Peet.exp.stability}, Peet proves that exponentially
stable polynomial systems have polynomial Lyapunov functions on
bounded regions. In~\cite{Peet.Antonis.converse.sos.CDC},\cite{Peet.Antonis.converse.sos.journal},
Peet and Papachristodoulou provide a degree bound for this Lyapunov function (depending on decay rate of trajectories) and show that it can be made sos.
% depending on properties of the trajectories such as the decay rate. As a byproduct of their proof, they show that the constructed Lyapunov function is a sum of squares.
%
%though this statement alone also has a simpler proof using the previous result %of Peet~\cite{Peet.exp.stability}; see Lemma~\ref{lem:W=V^2}. 
%There is no guarantee in~\cite{Peet.Antonis.converse.sos.CDC},\cite{Peet.Antonis.converse.sos.journal}, however, that the inequality on the derivative of the Lyapunov function has an sos certificate. 
The question whether the inequality on the derivative of the Lyapunov function can also have an sos certificate is not studied in these references.

%On the subject of degree bounds, the works of Nie and Schweighofer~\cite{Nie_Markus_complexity_Psatz} and Lombardi, Perrucci, and Roy~\cite{lombardi_degreebound_H17} (among others) provide upper bounds on the degree of sum of squares certificates of nonnegativity of polynomials on basic semialgebraic sets.
%
%
%As we remarked before, there is
%therefore no guarantee that the SDP will find such a Lyapunov function as the replacement of the inequality with sos can make the constraint stronger. (See our counterexample in Section~\ref{sec:the.counterexample} where this very phenomenon
%occurs.) 
%
%In the setting of switched dynamical systems, a related result which we in fact utilize in this paper is due to Mason et al.~\cite{switch.common.poly.Lyap}, where they show that asymptotically stable switched systems in continuous time admit polynomial Lyapunov functions. 
%
%Other relevant work in the control literature include the work of
%Chesi in \cite{Chesi.psd.not.sos} where the gap between positive
%and sos polynomials is studied.
%
%, and the work of Peyrl and Parrilo
%in \cite{ Peyrl.Pablo.SOS.dual.Lyap} where it is shown that
%certificates of infeasibility of the SDP in
%(\ref{eq:V.SOS}), (\ref{eq:-Vdot.SOS}) can be extracted from
%equilibria, orbits, or unbounded solutions. A preliminary version of the current paper has appeared in~\cite{AAA_PP_CDC11_converseSOS_Lyap}.

\section{Preliminaries}\label{sec:prelim}
Throughout the paper, we will be concerned with Lyapunov functions
that are (multivariate) polynomials. We say that a polynomial function $V:\mathbb{R}^n\rightarrow\mathbb{R}$ is \emph{nonnegative} if $V(x)\geq0$ for all
$x$, \emph{positive definite} if
$V(x)>0$ for all $x\neq 0$, \emph{negative definite} if $-V$ is positive definite, and \emph{radially unbounded} if $V(x)\rightarrow\infty$ as $||x||\rightarrow\infty$. A polynomial $V$ of degree $d$ is said to be
\emph{homogeneous} if it satisfies $V(\lambda x)=\lambda^d V(x)$
for any scalar $\lambda\in\mathbb{R}$. This condition holds if and
only if all monomials of $V$ have degree $d$. A homogeneous polynomial is also called a \emph{form}. The \emph{top homogeneous component} of a polynomial $p$ is the homogeneous polynomial formed by the collection of the highest order monomials of $p$.
%A polynomial is \emph{homogeneous} if all of its monomials have
%the same degree. A homogeneous polynomial $V$ of degree $d$
%satisfies $V(\lambda x)=\lambda^d V(x)$ for any scalar
%$\lambda\in\mathbb{R}$.

We say that a polynomial $V$ is a \emph{sum of squares} (sos) if
$V=\sum_i^m q_i^2$ for some polynomials $q_i$. We do not present
here the SDP that decides if a given polynomial
is sos since it has already appeared in several places. The
unfamiliar reader is referred to~\cite{sdprelax}. If $V$ is sos,
then $V$ is nonnegative, though the converse is in general not true~\cite{Hilbert_1888},~\cite{Reznick}.

\section{Converse SOS Lyapunov theorems}\label{sec:converse.sos.results}
%The results in this section are specific homogeneous systems.

We start by observing that existence of a polynomial Lyapunov
function immediately implies existence of a Lyapunov function that
is a sum of squares.

\begin{lemma}\label{lem:W=V^2}
Given a polynomial vector field, suppose there exists
a polynomial Lyapunov function $V$ such that $V$ and
$-\dot{V}$ are positive definite. Then, there also exists a
polynomial Lyapunov function $W$ such that $W$ and $-\dot{W}$ are positive definite and $W$ is sos.
%If a polynomial dynamical system has a positive definite
%polynomial Lyapunov function $V$ with a negative definite
%derivative $\dot{V}$, then it also admits a positive definite
%polynomial Lyapunov function $W$ which is a sum of squares.
\end{lemma}
\begin{proof}
Take $W=V^2$. Then $W$ and $-\dot{W}=-2V\dot{V}$ are positive definite and $W$ is sos (though $-\dot{W}$ may not be sos).
\end{proof}

We will next prove a result that guarantees the derivative of the
Lyapunov function will also satisfy the sos condition, though this
result is restricted to homogeneous systems.

A polynomial vector field $\dot{x}=f(x)$
%\begin{equation}\nonumber %\label{eq:xdot=f(x)}
%\dot{x}=f(x)
%\end{equation}
is \emph{homogeneous} if all entries of $f$ are homogeneous
polynomials of the same degree, i.e., if all the monomials in all
the entries of $f$ have the same degree. Homogeneous systems are extensively studied in the control literature; see
e.g.~\cite{Stability_homog_poly_ODE}, \cite{Stabilize_Homog},~\cite{homog.feedback},~\cite{HomogHomog},~\cite{homog.systems}, \cite{Baillieul_Homog_geometry}, and references therein.

%~\cite{Chesi_book}, \cite{Baillieul_Homog_geometry}, \cite{Cubic_Homog_Planar} and references therein.

%~\cite{Stability_homog_poly_ODE}, \cite{Stabilize_Homog}, \cite{homog.feedback}, \cite{Baillieul_Homog_geometry}, \cite{Cubic_Homog_Planar}, \cite{HomogHomog}, \cite{homog.systems}

 A basic fact about
homogeneous vector fields is that for these systems the notions of
local and global stability are equivalent. Indeed, a homogeneous
vector field of degree $d$ satisfies $f(\lambda x)=\lambda^d f(x)$
for any scalar $\lambda$, and therefore the value of $f$ on the
unit sphere determines its value everywhere. It is also well-known
that an asymptotically stable homogeneous system admits a
homogeneous Lyapunov function~\cite{HomogHomog}.
%Note that such a Lyapunov function is always radially unbounded.

We will use the following Positivstellens\"atze due to
Scheiderer to prove our converse sos Lyapunov theorem.

\begin{theorem}[Scheiderer,~\cite{Claus_Hilbert17}] \label{thm:claus}
Given any two positive definite homogeneous polynomials $p$ and
$q$, there exists an integer $k$ such that $pq^k$ is a sum of
squares.
\end{theorem}

\begin{theorem}\label{thm:poly.lyap.then.sos.lyap}
Given a homogeneous polynomial vector field, suppose there exists
a homogeneous polynomial Lyapunov function $V$ such that $V$ and
$-\dot{V}$ are positive definite. Then, there also exists a
homogeneous polynomial Lyapunov function $W$ such that $W$ and $-\dot{W}$ are both sos (and positive definite).
%is sos
%and $-\dot{W}$ is sos.
\end{theorem}

\begin{proof}
Observe that $V^2$ and $-2V\dot{V}$ are both positive definite and
homogeneous polynomials. (Homogeneous polynomials are closed under products and the gradient of a homogeneous polynomial has homogeneous entries.) Applying Theorem~\ref{thm:claus} to these
two polynomials, we conclude that there exists an integer $k$ such
that $(-2V\dot{V})(V^2)^k$ is sos. Let $$W=V^{2k+2}.$$ Then, $W$
is clearly sos since it is a perfect even power. Moreover,
$$-\dot{W}=-(2k+2)V^{2k+1}\dot{V}=-(k+1)2V^{2k}V\dot{V}$$
is also sos by the previous claim. Positive definiteness of $W$ and $-\dot{W}$ is clear from the construction.
%\footnote{Note that $W$ and
%$-\dot{W}$ are positive definite and therefore prove global
%asymptotic stability.}
\end{proof}

The polynomial $W$ constructed in the proof above has higher degree than the polynomial $V$, though it seems difficult to construct vector fields for which this degree increase is necessary. In~\cite{AAA_ACC12_Cubic_Difficult}, we show that unless P=NP, there cannot be a polynomial time algorithm to test whether a cubic homogeneous vector field is globally asymptotically stable, or even to test whether it admits a quadratic Lyapunov function. Hence, just from complexity considerations, we expect there to exist GAS cubic vector fields (of possibly high dimensions) that admit a degree-2 polynomial Lyapunov function satisfying (\ref{eq:V.positive})-(\ref{eq:Vdot.negative}), but for which the minimum degree of a polynomial satisfying the sos constraints in (\ref{eq:V.SOS})-(\ref{eq:-Vdot.SOS}) is arbitrarily high. One difficulty with explicitly constructing such examples stems from non-uniqueness of Lyapunov functions. This makes it insufficient to simply engineer $V$ or $-\dot{V}$ to be ``positive but not sos''; one needs to show that \emph{any} polynomial Lyapunov function of a given degree fails to satisfy the sos constraints in (\ref{eq:V.SOS})-(\ref{eq:-Vdot.SOS}).

Next, we develop a theorem that removes the homogeneity
assumption from the vector field in Theorem~\ref{thm:poly.lyap.then.sos.lyap}, but instead is restricted to
vector fields on the plane. For this, we need another result of
Scheiderer.

\begin{theorem}[{Scheiderer,~\cite[Cor. 3.12]{Claus_3vars_sos}}] \label{thm:claus.3vars}
Let $p\mathrel{\mathop:}=p(x_1,x_2,x_3)$ and
$q\mathrel{\mathop:}=q(x_1,x_2,x_3)$ be two homogeneous
polynomials in three variables, with $p$ positive semidefinite and
$q$ positive definite. Then, there exists an integer $k$ such that
$pq^k$ is a sum of squares.
\end{theorem}

\begin{theorem}\label{thm:poly.lyap.then.sos.lyap.PLANAR}
Given a (not necessarily homogeneous) polynomial vector field in
two variables, suppose there exists a positive definite polynomial
Lyapunov function $V,$ with $-\dot{V}$ positive definite, and such
that the top homogeneous component of $V$ has no zeros\footnote{This
requirement is only slightly stronger than the requirement of
radial unboundedness, which is imposed on $V$ by Lyapunov's
theorem anyway. See Section~\ref{sec.h.o.t}.}. Then, there also exists a polynomial Lyapunov
function $W$ such that $W$ and $-\dot{W}$ are both sos (and positive definite).
\end{theorem}

\begin{proof}
Let $\tilde{V}=V+1$. So, $\dot{\tilde{V}}=\dot{V}$. Let us denote the degrees
of $\tilde{V}$ and $\dot{\tilde{V}}$ by $d_1$ and $d_2$ respectively. Consider the
(non-homogeneous) polynomials $\tilde{V}^2$ and
$-2\tilde{V}\dot{\tilde{V}}$ in the variables
$x\mathrel{\mathop:}=(x_1,x_2)$. Note that
$\tilde{V}^2$ is nowhere zero and $-2\tilde{V}\dot{\tilde{V}}$ is
only zero at the origin. Our first step is to homogenize these
polynomials by introducing a new variable $y$. Observing that the
homogenization of products of polynomials equals the product of
their homogenizations, we obtain the following two trivariate forms:
\begin{equation}\label{eq:V^2.homoegenized}
y^{2d_1}\tilde{V}^2(\textstyle{\frac{x}{y}}),
\end{equation}
\begin{equation}\label{eq:-2V.Vdot.homogenized}
-2y^{d_1}y^{d_2}\tilde{V}(\textstyle{\frac{x}{y}})\dot{\tilde{V}}(\textstyle{\frac{x}{y}}).
\end{equation}
Since by assumption the highest order term of $V$ has no zeros,
the form in (\ref{eq:V^2.homoegenized}) is positive definite. The
form in (\ref{eq:-2V.Vdot.homogenized}), however, is only positive
semidefinite. In particular, since $\dot{\tilde{V}}=\dot{V}$ has
to vanish at the origin, the form in
(\ref{eq:-2V.Vdot.homogenized}) has a zero at the point
$(x_1,x_2,y)=(0,0,1)$. Nevertheless, since
Theorem~\ref{thm:claus.3vars} allows for positive semidefiniteness
of one of the two forms, by applying it to the forms in
(\ref{eq:V^2.homoegenized}) and (\ref{eq:-2V.Vdot.homogenized}),
we conclude that there exists an integer $k$ such that
\begin{equation}\label{eq:-2V.Vdot.homog.*.V^2.homog^k}
-2y^{d_1(2k+1)}y^{d_2}\tilde{V}(\textstyle{\frac{x}{y}})\dot{\tilde{V}}(\textstyle{\frac{x}{y}})\tilde{V}^{2k}(\textstyle{\frac{x}{y}})
\end{equation}
is sos. Let $W=\tilde{V}^{2k+2}.$ Then, $W$ is clearly sos.
Moreover,
$$-\dot{W}=-(2k+2)\tilde{V}^{2k+1}\dot{\tilde{V}}=-(k+1)2\tilde{V}^{2k}\tilde{V}\dot{\tilde{V}}$$
is also sos because this polynomial is obtained from
(\ref{eq:-2V.Vdot.homog.*.V^2.homog^k}) by setting
$y=1$. Positive definiteness of $W$ and $-\dot{W}$ is again clear from the construction. Note that while $W$ does not vanish at the origin, it achieves its minimum there, and hence provides a proof of asymptotic stability.
%Note that $W$ achieves its minimum at the origin (though it does not vanish at the origin), which is what is needed to prove asymptotic stability.

%\footnote{Once again, we note that the function $W$
%constructed in this proof is radially unbounded, achieves its
%global minimum at the origin, and has $-\dot{W}$ positive
%definite. Therefore, $W$ proves global asymptotic stability.}
\end{proof}

The polynomial $W$ constructed in the proof above again has higher degree than the polynomial $V$. The example below shows that such a degree increase is sometimes necessary.

\begin{theorem}[see~\cite{AAA_PP_CDC11_converseSOS_Lyap}]
The vector field
\begin{equation} \nonumber
\begin{array}{lll}
\dot{x_{1}}&=&-x_1^3x_2^2+2x_1^3x_2-x_1^3+4x_1^2x_2^2-8x_1^2x_2+4x_1^2 \\
\ &\ &-x_1x_2^4+4x_1x_2^3-4x_1+10x_2^2
\\
\dot{x_{2}}&=&-9x_1^2x_2+10x_1^2+2x_1x_2^3-8x_1x_2^2-4x_1-x_2^3
\\
\ &\ &+4x_2^2-4x_2
\end{array}
\end{equation}
admits a degree-2 polynomial Lyapunov function that proves its global asymptotic stability. However, the minimum degree of a polynomial Lyapunov function that satisfies the sos constraints in (\ref{eq:V.SOS})-(\ref{eq:-Vdot.SOS}) is equal to 4.
\end{theorem}

%\subsection{Existence of sos Lyapunov functions for switched linear
%systems}\label{subsec:extension.to.switched.sys}

\subsection{SOS certificates for stability of switched linear
systems}\label{subsec:extension.to.switched.sys}

The result of Theorem~\ref{thm:poly.lyap.then.sos.lyap} extends in a
straightforward manner to Lyapunov analysis of switched systems.
In particular, we are interested in the highly-studied problem of establishing stability of arbitrary switched linear systems:
\begin{equation}\label{eq:switched.linear.system}
\dot{x}=A_i x, \quad i\in\{1,\ldots,m\},
\end{equation}
$A_i\in\mathbb{R}^{n\times n}$. We assume the minimum dwell time
of the system (i.e., the minimum time between two consecutive
switches) is bounded away from zero. This guarantees that the
solutions of (\ref{eq:switched.linear.system}) are well-defined.
%Existence of a common Lyapunov function is necessary and
%sufficient for (global) asymptotic stability under arbitrary
%switching (ASUAS) of system (\ref{eq:switched.linear.system}). 
The (global) asymptotic stability under arbitrary switching (ASUAS) of system (\ref{eq:switched.linear.system}) is equivalent to
asymptotic stability of the linear differential inclusion
\begin{equation}\nonumber
\dot{x}\in co\{A_i\}x, \quad i\in\{1,\ldots,m\},
\end{equation}
where $co$ here denotes the convex hull.
%This is obvious in view
%of the fact that a common Lyapunov function decreasing with
%respect to the matrices $A_i$ would also decrease with respect to
%any matrix in the convex hull.
%
%It is also known that ASUAS of (\ref{eq:switched.linear.system})
%is equivalent to exponential stability under arbitrary
%switching~\cite{Angeli.homog.switched}. 
%
A common approach for
analyzing stability of these systems is to use the sos
technique to search for a common polynomial Lyapunov
function~\cite{PraP03},\cite{Chest.et.al.sos.robust.stability}. We
will prove the following result.

\begin{theorem}\label{thm:converse.sos.switched.sys}
The switched linear system in (\ref{eq:switched.linear.system}) is
asymptotically stable under arbitrary switching if and only if
there exists a common homogeneous polynomial Lyapunov function $W$
such that
\begin{equation}\nonumber
\begin{array}{rl}
W & \mbox{sos} \\
-\dot{W}_i=-\langle \nabla W(x), A_ix\rangle & \mbox{sos},
\end{array}
\end{equation}
for $i=1,\ldots, m$, where the polynomials $W$ and $-\dot{W}_i$
are all positive definite.
\end{theorem}
To prove this result, we make use of the following theorem of Mason
et al.

\begin{theorem}[Mason et al.,~\cite{switch.common.poly.Lyap}] \label{thm:switch.poly.exists} If the switched linear system
in (\ref{eq:switched.linear.system}) is asymptotically stable
under arbitrary switching, then there exists a common homogeneous
polynomial Lyapunov function $V$ such that
\begin{equation}\nonumber
\begin{array}{rll}
V &>&0 \ \ \forall x\neq 0  \\
-\dot{V}_i(x)=-\langle \nabla V(x), A_ix\rangle &>&0 \ \  \forall
x\neq 0,
\end{array}
\end{equation}
for $i=1,\ldots,m$.
\end{theorem}

The next proposition is an extension of
Theorem~\ref{thm:poly.lyap.then.sos.lyap} to switched systems (not
necessarily linear).

\begin{proposition}\label{prop:switch.poly.then.sos.poly}
Consider an arbitrary switched dynamical system
\begin{equation}\nonumber %\label{eq:switched.polynomial.system}
\dot{x}=f_i(x), \quad i\in\{1,\ldots,m\},
\end{equation}
where $f_i(x)$ is a homogeneous polynomial vector field of degree
$d_i$ (the degrees of the different vector fields can be
different). Suppose there exists a common positive definite
homogeneous polynomial Lyapunov function $V$ such that
$$-\dot{V}_i(x)=-\langle \nabla V(x), f_i(x)\rangle$$
is positive definite for all $i\in\{1,\ldots, m\}$. Then there
exists a common homogeneous polynomial Lyapunov function $W$ such
that $W$ is sos (and positive definite) and the polynomials $$-\dot{W}_i=-\langle \nabla
W(x), f_i(x)\rangle,$$ for all $i\in\{1,\ldots, m\}$, are also
sos (and positive definite).
\end{proposition}

\begin{proof}
Observe that for each $i$, the polynomials $V^2$ and
$-2V\dot{V}_i$ are both positive definite and homogeneous.
Applying Theorem~\ref{thm:claus} $m$ times to these pairs of
polynomials, we conclude that there exist positive integers $k_i$
such that
\begin{equation}\label{eq:-2VVdotV^2V^k_i.sos}
(-2V\dot{V}_i)(V^2)^{k_i} \  \mbox{is sos,}
\end{equation}
for $i=1,\ldots,m$. Let $$k=\max\{k_1,\ldots,k_m\},$$ and let
$W=V^{2k+2}.$ Then, $W$ is clearly sos. Moreover, for each $i$,
the polynomial
\begin{equation}\nonumber
\begin{array}{rll}
-\dot{W}_i &=&-(2k+2)V^{2k+1}\dot{V}_i \\
\ &=&-(k+1)2V\dot{V}_iV^{2k_i}V^{2(k-k_i)}
\end{array}
\end{equation}
is sos since $(-2V\dot{V}_i)(V^{2k_i})$ is sos by
(\ref{eq:-2VVdotV^2V^k_i.sos}), $V^{2(k-k_i)}$ is sos as an even
power, and products of sos polynomials are sos.
\end{proof}

The proof of Theorem~\ref{thm:converse.sos.switched.sys} now
simply follows from Theorem~\ref{thm:switch.poly.exists} and
Proposition~\ref{prop:switch.poly.then.sos.poly} in the special
case where $d_i=1$ for all $i$.

%The analogue of
%Theorem~\ref{thm:converse.sos.switched.sys} in discrete time has been proven by %Parrilo and Jadbabaie in~\cite{Pablo_Jadbabaie_JSR_journal} using different techniques.

% It is shown that if
%(\ref{eq:switched.linear.system.in.DT}) is asymptotically stable
%under arbitrary switching, then there exists a homogeneous
%polynomial Lyapunov function $W$ such that
%\begin{equation}\nonumber
%\begin{array}{rl}
%W(x) & \mbox{sos} \\
%W(x)-W(A_ix) & \mbox{sos},
%\end{array}
%\end{equation}
%for $i=1,\ldots,m$.

%\section{The conservatism of ``$V$ sos''}\label{sec:V.sos.conservatism}

%In this section, we describe situations in which the requirement of the Lyapunov function being a sum of squares introduces no conservatism. Recall from Lemma~\ref{lem:W=V^2} that one can always make a polynomial Lyapunov function to be sos by going from $V$ to $V^2$. However, in scenarios described in this section, one does not need to pay the computational price associated with this degree increase.

\subsection{The ``$V $sos'' requirement for switched linear systems}\label{subsec:switched.V.always.sos}

Our next two propositions show that for switched linear systems, both in discrete time and in continuous time, the sos condition on the Lyapunov function
itself is never conservative, in the sense that if one of the
``decrease inequalities'' has an sos certificate, then the Lyapunov function is
automatically sos. These propositions are really statements about
linear systems, so we will present them as such. However, since
stable linear systems always admit (sos) quadratic Lyapunov functions,
the propositions are only interesting in the context where a
common polynomial Lyapunov function for a switched linear system
is seeked.

\begin{proposition}\label{prop:switch.DT.V.automa.sos}
Consider the linear dynamical system $x_{k+1}=Ax_k$ in discrete
time. Suppose there exists a positive definite polynomial Lyapunov
function $V$ such that $V(x)-V(Ax)$ is positive definite and sos.
Then, $V$ is sos.
\end{proposition}

\begin{proof}
Consider the polynomial $V(x)-V(Ax)$ that is sos by assumption. If
we replace $x$ by $Ax$ in this polynomial, we conclude that the
polynomial $V(Ax)-V(A^2 x)$ is also sos. Hence, by adding these
two sos polynomials, we get that $V(x)-V(A^2x)$ is sos. This
procedure can be repeated to infer that for any integer
$k\geq 1$, the polynomial
\begin{equation}\label{eq:V-V(A^k)}
V(x)-V(A^kx)
\end{equation}
is sos. Since by assumption $V$ and $V(x)-V(Ax)$ are positive
definite, the linear system must be GAS, and hence $A^k$ converges
to the zero matrix as $k\rightarrow\infty$. Observe that for all
$k$, the polynomials in (\ref{eq:V-V(A^k)}) have degree equal to
the degree of $V$, and that the coefficients of $V(x)-V(A^kx)$
converge to the coefficients of $V-V(0)$ as $k\rightarrow\infty$. Since
for a fixed degree and dimension the cone of sos polynomials is
closed~\cite{RobinsonSOS}, it follows that $V-V(0)$ is sos. Hence, $V$ is sos.
\end{proof}

Similarly, in continuous time, we have the following:
\begin{proposition}\label{prop:switch.CT.V.automa.sos}
Consider the linear dynamical system $\dot{x}=Ax$ in continuous
time. Suppose there exists a positive definite polynomial Lyapunov
function $V$ such that $-\dot{V}=-\langle \nabla V(x), Ax \rangle$
is positive definite and sos. Then, $V$ is sos.
\end{proposition}
\begin{proof}
The value of the polynomial $V$ along the trajectories of the
dynamical system satisfies the relation
$$V(x(t))=V(x(0))+\int_o^t \dot{V}(x(\tau)) d_\tau.$$ Since the
assumptions imply that the system is GAS, $x(t)\rightarrow 0$
as $t$ goes to infinity. By evaluating the
above equation at $t=\infty$, rearranging terms, and substituting
$e^{A\tau}x$ for the solution of the linear system at time $\tau$
starting at initial condition $x$, we obtain
$$V(x)=\int_0^\infty -\dot{V}(e^{A\tau}x) d_\tau+V(0).$$ By assumption, $-\dot{V}$ is sos and therefore for any value of $\tau$, the integrand
$-\dot{V}(e^{A\tau}x)$ is an sos polynomial. Since converging
integrals of sos polynomials are sos, it follows that $V$ is sos.
\end{proof}

%\footnote{One way to see this is to think of the representation of the integral as the limit of its Riemann sum. Once again, we have a converging sequence of sos polynomials in fixed dimension and degree and can utilize the closedness of the sos cone.}

One may wonder if a similar statement holds for nonlinear vector fields? The answer is negative.

%\begin{remark}
%The previous proposition does not hold in general if the system is not
%linear. For example, consider any positive form $V$ that is not a
%sum of squares and define a dynamical system by $\dot{x}=-\nabla
%V(x)$. In this case, both $V$ and $-\dot{V}=||\nabla V(x)||^2$ are
%positive definite and $-\dot{V}$ is sos, though $V$ is not sos.
%\end{remark}

\begin{lemma}\label{lem:V.autom.sos.not.true.NL}
There exist a polynomial vector field $\dot{x}=f(x)$ and a polynomial Lyapunov function $V,$ such that $V$ and $-\dot{V}$ are positive definite, $-\dot{V}$ is sos, but $V$ is not sos.
\end{lemma}

\begin{proof}
Consider any positive form $V$ that is not a
sum of squares. (An example is $x_1^4x_2^2+x_1^2x_2^4-3x_1^2x_2^2x_3^2+x_3^6+\frac{1}{250}(x_1^2+x_2^2+x_3^2)^3$.) Define a dynamical system by $$\dot{x}=-\nabla
V(x).$$ In this case, both $V$ and $-\dot{V}=||\nabla V(x)||^2$ are
positive definite and $-\dot{V}$ is sos, though $V$ is not sos. To see that $-\dot{V}$ is positive definite, note that a homogeneous function $V$ of degree $d$ satisfies the Euler identity: $V(x)=\frac{1}{d}x^T\nabla V(x)$. If we had $-\dot{V}(x)=0$ for some $x\neq 0$, then we would have $\nabla V(x)=0$ and hence also $V(x)=0$, which is a contradiction.
\end{proof}

\section{Working with the top homogeneous component of $V$}\label{sec.h.o.t}

We show in this final section that for global stability analysis with sos techniques, the requirement of the polynomial Lyapunov function being sos can be replaced with the requirement of its top homogeneous component being sos. We also show that doing so has a number of advantages.
%
%there are a number of advantages in requiring the top homogeneous component of a candidate polynomial Lyapunov function to be sos, as opposed to asking the full polynomial to be sos. 
The point of departure is the following proposition, which states that in presence of radial unboundedness, the positivity requirement of the Lyapunov function is not needed.

\begin{proposition}\label{prop:hot.implies.gas}
Consider the vector field (\ref{eq:CT.dynamics}). If there exists a continuously differentiable, radially unbounded Lyapunov function $V$ that satisfies $\dot{V}(x)<0$, $\forall x\neq 0$, then the origin is globally asymptotically stable.
\end{proposition}

\begin{proof}
We first observe that since $V$ is radially unbounded and continuous, it must be lower bounded. In fact, radial unboundedness implies that the set $\{x|\ V(x)\leq V(1,0,\ldots,0)\}$ is compact, and continuity of $V$ implies that the minimum of $V$ on this set, which equals the minimum of $V$ everywhere, is achieved. We further claim that this minimum can only be achieved at the origin. Suppose there was a point $\bar{x}\neq 0$ that was a global minimum for $V$. As a necessary condition of global optimality, we must have $\nabla V(\bar{x})=0$. But this implies that $\dot{V}(\bar{x})=0$, which is a contradiction. Now consider a new Lyapunov function $W$ defined as $W(x)\mathrel{\mathop:=V(x)-V(0)}$. Then $W$ is positive definite, radially unbounded, and has $\dot{W}=\dot{V}<0$, for all $x\neq 0$. Hence, $W$ satisfies all assumptions of Lyapunov's theorem (see, e.g.,~\cite[Chap. 4]{Khalil:3rd.Ed}), therefore implying global asymptotic stability.\footnote{The conditions of Proposition~\ref{prop:hot.implies.gas} do not imply that $V(0)=0$ as is customary for Lyapunov functions. However, what we really need is for $V$ to attain its global minimum at the origin, which is the case here.}
\end{proof}

A sufficient condition for a polynomial $p$ to be radially unbounded is for its \emph{top homogeneous component}, denoted by $t.h.c.(p)$, to form a positive definite polynomial. This condition is almost necessary: radial unboundedness of $p$ implies that $t.h.c.(p)$ needs to be positive semidefinite. Since we seek for radially unbounded Lyapunov functions anyway, this suggests that in our SDP search for Lyapunov functions, we can replace the conditions ``$V$ sos and $-\dot{V}$ sos'', with ``$t.h.c.(V)$ sos and $-\dot{V}$ sos''. The following lemma tells us that this can only help us in terms of conservatism.

%\begin{lemma}\label{lem:hot.is.no.more.conservative}
%Let $V\mathrel{\mathop:}=V(x_1,\ldots,x_n)$ be a polynomial. If $V$ is sos, then $t.h.c.(V)$ is sos. Further, the converse statement is not true (even if $V$ is nonnegative).
%\end{lemma}

\begin{lemma}\label{lem:hot.is.no.more.conservative}
Consider the vector field (\ref{eq:CT.dynamics}) and suppose there exists a polynomial function $V\mathrel{\mathop:}=V(x_1,\ldots,x_n)$ that makes both $t.h.c.(V)$ and $-\dot{V}$ sos (and positive definite), then the origin is GAS. Moreover the condition $t.h.c.(V)$ sos is never more conservative (and can potentially be less conservative) than the condition $t.h.c.(V)$ sos.
%Let $V\mathrel{\mathop:}=V(x_1,\ldots,x_n)$ be a polynomial. If $V$ is sos, then $t.h.c.(V)$ is sos. Further, the converse statement is not true (even if $V$ is nonnegative).
\end{lemma}

\begin{proof}
The claim of global asymptotic stability follows from Proposition~\ref{prop:hot.implies.gas}. To prove the claim about conservatism, we show that if $V$ is sos, then $t.h.c.(V)$ is sos, while the converse is not true (even if $V$ is nonnegative). Let $d$ be the degree of $V$ and consider the standard homogenization $V_h$ of $V$ in $n+1$ variables given by $V_h(x_1,\ldots,x_1,y)=y^dV(\frac{x}{y})$. Since $V$ is sos, $V_h$ must be sos~\cite{Reznick}. But this implies that $t.h.c.(V)$ must be sos since it is a restriction of $V_h$: $t.h.c.(V)=V_h(x_1,\ldots,x_n,0)$. A counterexample to the converse is the Motzkin polynomial $$x_1^4x_2^2+x_1^2x_2^4-3x_1^2x_2^2+1,$$ whose top homogeneous component $x_1^4x_2^2+x_1^2x_2^4$ is sos, but the polynomial itself is not sos even though it is nonnegative~\cite{Reznick}.
\end{proof}

In the particular case of the plane ($n=2$), the constraint ``$t.h.c.(V)$ sos'' in fact introduces no conservatism at all. This is because all nonnegative bivariate homogeneous polynomials (of any degree) are sums of squares~\cite{Reznick}. By contrast, the alternative condition ``$V$ sos'' can be conservative as there are (non-homogeneous) nonnegative bivariate polynomials that are not sums of squares. In addition to these features in terms of conservatism, the requirement ``$t.h.c.(V)$ sos'' is a cheaper constraint to impose than ``$V$ sos''. Indeed, if $V$ is an $n$-variate polynomial of degree $2d$, simple calculation shows that the SDP underlying the former condition has  $\frac{2d(n+2d-1)!}{2d!n!}$ fewer equality constraints and $\frac{d(n+d-1)!((2n+d)(n+d-1)!+d!n!)}{2(d!n!)^2}$ fewer decision variables than the latter.
% $t.h.c.(V)$ has ${n+2d \choose 2d}-{n+2d-1 \choose 2d}$ fewer coefficients than $V$.
Let us end with a concrete example.

\begin{example}\label{ex:thc}
Consider the vector field 
\begin{equation}\nonumber
\begin{array}{l}
\dot{x}_1= 0.36x_1+2x_2-0.32x_1^7-0.02x_1x_2^6+8x_2^7+3x_1^2x_2^5\\
\dot{x}_2= -2x_1-0.44x_2-16x_1^7-x_1x_2^6-0.16x_2^7-0.06x_1^2x_2^5.
\end{array}
\end{equation}
We solve an SDP which searches for a polynomial Lyapunov function $V$ of degree $d$ that satisfies ``$t.h.c.(V)$ sos and $-\dot{V}$ sos''. This SDP is infeasible for $d=2,4,6$. Because of Lemma~\ref{lem:hot.is.no.more.conservative}, we know that the SDP with the more standard conditions ``$V$ sos and $-\dot{V}$ sos'' would have been infeasible also. Moreover, since nonnegative bivariate forms are sos, we know that our constraint ``$t.h.c.(V)$ sos'' is lossless (while we wouldn't be able to make such a claim for the constraint ``$V$ sos'').

For $d=8$, our SDP returns the following solution:
\begin{equation}\nonumber
\begin{array}{l}
%V(x)= 2.1948x_1^2+4.2373x_1^8+2.1830x_2^2+2.1697x_2^8\\ \  +1.0546x_1^2x_2^6+0.8632x_1x_2-0.2862x_1^7x_2+0.0374x_1^6x_2^2\\ \  +0.0423x_1^5x_2^3-0.0107x_1^4x_2^4-0.0211x_1^3x_2^5+0.0393x_1x_2^7
V(x)= 2.195x_1^2+4.237x_1^8+2.183x_2^2+2.170x_2^8\\ \  +1.055x_1^2x_2^6+0.863x_1x_2-0.286x_1^7x_2+0.037x_1^6x_2^2\\ \  +0.042x_1^5x_2^3-0.011x_1^4x_2^4-0.021x_1^3x_2^5+0.039x_1x_2^7.
\end{array}
\end{equation}
By checking the eigenvalues of the Gram matrices in our SDP, we can see that the top homogeneous component of this polynomial, as well as $-\dot{V},$ are positive definite. Hence, by Lemma~\ref{lem:hot.is.no.more.conservative}, we know that the system is GAS. For this example it happens that the returned solution $V$ is sos although we asked for a weaker and cheaper condition. Indeed, the semidefinite constraint that we imposed to require $t.h.c.(V)$ to be sos has 105 fewer decision variables and 36 fewer equality constraints than the semidefinite constraint needed to impose $V$ sos. 
%
%
%radial unboundedness implies positive definiteness of our Lyapunov function as we must have $V(0)=0$ and the global minimum of $V$ being ach
%
%our Lyapunov function 
%
%
%vanishes at the origin and, as shown in the proof of Proposition~\ref{prop:hot.implies.gas}, achieves its global minimum at the origin.
\end{example}
%\begin{remark}
%The condition ``$t.h.c.(V)$ sos may in general return Lyapunov functions
%\end{remark}

In general, by leaving out the constant term in the parametrization of the polynomial $V$, we can make sure that $V(0)=0$ and hence (in view of the proof of Proposition~\ref{prop:hot.implies.gas}) radial unboundedness would automatically imply positive definiteness of our Lyapunov function.

\paragraph{Acknowledgments.} The authors are grateful to Claus Scheiderer for insightful discussions.

\bibliographystyle{abbrv}
\bibliography{pablo_amirali}
%\end{multicols}

\end{document}